%% file: main_new.tex
\documentclass{article}
\usepackage[utf8]{inputenc}
\include{packages}

\newcommand{\abs}[1]{\left\vert#1\right\vert}

\usepackage{authblk}
\title{\sc Propagation reversal on trees in the large diffusion regime}
\author[1]{Hermen Jan Hupkes \thanks{\tt hhupkes@math.leidenuniv.nl}}
\author[1,2]{Mia Juki\'c \thanks{\tt mia.jukic@tno.nl}}
\affil[1]{\small Mathematisch Instituut, Universiteit Leiden, P.O. Box 9512, 2300 RA Leiden, The Netherlands}
\affil[2]{\small Electromagnetic Signatures \& Propagation, Netherlands Organisation for Applied Scientific Research (TNO), 2509 JG The Hague, The Netherlands}

\title{\sc Propagation reversal on trees in the large diffusion regime}

\begin{document}
\maketitle
\begin{abstract}
    In this work we study  travelling wave solutions to  bistable reaction diffusion equations on bi-infinite $k$-ary trees in
    the continuum regime where the diffusion parameter is large.
    Adapting the spectral convergence method developed by Bates and his coworkers, 
    we  find an asymptotic prediction for the speed of 
    travelling front solutions. In addition, we prove that 
    the associated profiles converge to the solutions of
    a suitable limiting reaction-diffusion PDE. Finally, 
    for the standard cubic nonlinearity we provide explicit formula's
    to bound the thin region in parameter space where
    the propagation direction undergoes a reversal. 
\end{abstract}



\section{Introduction}
\label{sec:int}
In this work we study the reaction-diffusion-advection lattice  differential equation (LDE)
\begin{equation}\label{eqn:intro:LDE}
\begin{aligned}
    \dot{u}_i &= d\Big(ku_{i+1} - (k+1)u_i + u_{i-1}\Big) + g(u_i;a) ,
\qquad i\in \Z. 
\end{aligned}
\end{equation}
As explained below, we view the (real-valued) parameter $k >1$ as a branching factor. In addition, $d > 0$ encodes the \textit{diffusion strength} and the nonlinearity $g$ is of bistable type, for example 
\begin{equation}\label{eqn:intro:cubic}
g(u;a) = u(1-u)(u-a), \qquad  a \in (0,1).
\end{equation} 
This LDE is known \cite{Mallet-Paret1999} to admit travelling front solutions, i.e. solutions
of the form
\begin{equation}\label{eqn:intro:ansatz}
    u_i(t) = \Phi( i -ct), \qquad \qquad
    \Phi(-\infty) = 0, \qquad \qquad \Phi(+\infty) = 1. 
\end{equation}
In this paper we apply a version of the `spectral convergence' technique 
to study the behaviour of the pair $(c,\Phi)$ in the regime
where $d$ is large, the so-called continuum regime. In particular, our results supplement our earlier work \cite{Hupkes2022PropagationRF}, 
where we studied the small and intermediate $d$ regime using an
entirely different set of techniques.

\paragraph{Dynamics on $k$-ary trees}
Our primary motivation to study~\eqref{eqn:intro:LDE} is that 
the wavefronts \eqref{eqn:intro:ansatz} can be seen as
\textit{layer-wave} solutions of bistable reaction-diffusion equations posed on $k$-ary trees. In particular, the sign of the wavespeed $c$ determines which of the two stable roots of the nonlinearity $g$ can be expected to spread throughout the tree. We refer the reader to our previous work \cite{Hupkes2022PropagationRF} and a prior paper by Kouvaris, Kori and Mikhailov~\cite{Kouvaris2012} for numerical studies to support this claim and and an in-depth discussion of the potential application areas for our results.
Related results for monostable equations can be found in \cite{hoffman2016invasion}. Preliminary numerical investigations
show that for general trees, studying \eqref{eqn:intro:LDE}
with the \textit{average} branch-factor $k$ still has important
predictive capabilities.

For each fixed $k > 1$, the main results in~\cite{Hupkes2022PropagationRF} 
provide 
non-empty open sets of parameters $(a,d)$ where the expressions
$c > 0$, $c = 0$ respectively $c < 0$ are guaranteed to hold. In fact, upon steadily increasing the diffusion parameter $d >0$ for a fixed $a \sim 1$, waves transition from being pinned $(c = 0)$, travelling `down' the tree $(c > 0)$, being pinned once more $(c=0)$
towards finally travelling `up' the tree $(c < 0)$. This is illustrated
by the numerical results in Fig. \ref{intro:figure}.

These results were achieved by constructing explicit sub- and super-solutions and invoking the comparison principle. Although the boundaries of these regions agree reasonably well with numerical observations for small $d$,
they are  - by construction - not `asymptotically precise'. In particular, they do not accurately capture the transition region between $c > 0$ and $c < 0$, which appears to converge to a curve as $d$ increases. The main purpose of our results here is to increase our understanding of this transition
region where propagation reversal occurs.

\paragraph{Continuum regime}
In order to gain some preliminary intuition into the
diffusion-driven propagation reversal discussed above,
we recall from \cite{Hupkes2022PropagationRF} that
the LDE~\eqref{eqn:intro:LDE} can also
be interpreted as a spatial discretization of the 
partial differential equation (PDE)
\begin{equation}\label{eqn:intro:PDE}
    u_t(x,t)= \nu u_{xx}(x,t) + \beta u_x(x,t) +g\big(u(x,t);a\big),\quad x\in\mathbb{R}, \quad t>0.
\end{equation}
Indeed, using the correspondence $u_i(t) \sim u(ih, t)$
and the relation
\begin{equation}
\label{eq:int:def:nu:beta}
    \nu = \frac{1}{2}(k+1) d h^2 , \qquad \qquad 
    \beta = (k-1) d h,
\end{equation}
the standard central difference schemes 
\begin{equation}
\label{eq;int:central:diff:schemes}
u_x(x,t) \sim \frac{1}{2h} \big(u(x+h,t)-u(x-h,t)\big),
\qquad 
u_{xx}(x,t) \sim \frac{1}{h^2}\big(u(x+h,t)+u(x-h,t) - 2u(x,t) \big)
\end{equation}
reduce \eqref{eqn:intro:PDE} back to \eqref{eqn:intro:LDE}. Vice-versa, expanding the shifted terms in \eqref{eqn:intro:LDE} up to order $O(h^3)$ 
leads directly to \eqref{eqn:intro:PDE}.

For the cubic \eqref{eqn:intro:cubic}, the PDE \eqref{eqn:intro:PDE} admits explicit travelling front solutions 
\begin{equation}
    u(x,t) = \frac{1}{2}\big[ 1 + \tanh\big( (x - \sigma t)/\sqrt{8 \nu } \big) \big],
    \qquad \qquad 
   \sigma=\sqrt{2\nu}\left(a-\frac{1}{2} \right) - \beta    .
\end{equation}
Upon introducing the appropriate speed-scaling $\sigma = ch$
and recalling the coefficients~\eqref{eq:int:def:nu:beta},
we readily obtain the asymptotic prediction
\begin{equation}
\label{eqn:int:exp:for:c}
    c  = \sqrt{(k+1) d }\left(a - \frac{1}{2}\right) - (k-1) d ,
\end{equation}
which changes sign at the critical value
\begin{equation}
\label{eq:int:crit:diff}
    d(a,k) =  \frac{(k+1)}{(k-1)^2}\left(a - \frac{1}{2}\right)^2.
\end{equation}
The numerical results in Figure \ref{intro:figure} illustrate
that this prediction retains its accuracy for intermediate values of $d$, corresponding to values of $k$ that are far removed from the critical
regime $k \gtrsim 1$.

\paragraph{Spectral convergence}
Our main results provide a rigorous interpretation
and quantification for asymptotic predictions of this type.
This is achieved by using the spectral convergence approach
that was pioneered by Bates and his coworkers in
\cite{Bates2003}. The main 
feature of this approach is that
Fredholm properties of
linearized operators associated to travelling wave solutions
can be transferred from the spatially continuous setting of
\eqref{eqn:intro:PDE} to the spatially discrete
setting of \eqref{eqn:intro:LDE}.
The latter operators can then be used
in a standard fashion to close a fixed-point argument
and construct travelling front solutions to \eqref{eqn:intro:LDE}
that are close to those of their PDE counterpart
\eqref{eqn:intro:PDE}.

The main difficulty that needs to be overcome is that
this perturbation is highly singular, since (unbounded) derivatives
are replaced by difference quotients. As a consequence,
one must carefully work with weak limits
and use spectral properties of the (discrete) Laplacian together with
the bistable structure of the nonlinearity $g$
to counteract the loss of information that typically occurs
when using the weak topology.

In the present paper, the main additional complication
is the extra convective term appearing in the limiting
PDE  \eqref{eqn:intro:PDE} for $k > 1$. Indeed,
the coefficients $(\nu, \beta)$
introduced in \eqref{eq:int:def:nu:beta} scale differently
with respect to $h$. In particular, in our analysis
the associated extra terms cannot be seen as `small'
and must be handled with care by invoking the 
natural directional asymmetry that trees have.
A second main issue is that we want to obtain estimates
that are uniform in terms of the parameters $k$ and $a$. 
Such control is necessary in order to formulate quantitative results
for the critical curves $d(a,k)$.


\paragraph{Outlook}

Let us emphasize that we expect that our techniques
can be applied to a much larger class of problems
than the scalar setting of \eqref{eqn:intro:LDE}.
For example, following the framework developed in \cite{HJHFHNINFRANGE},
it should be possible to perform a similar analysis
for the FitzHugh-Nagumo system
and other multi-component reaction-diffusion problems.
Indeed, we do not use the comparison principle in this paper,
as opposed to the earlier work in \cite{Hupkes2022PropagationRF}.
We also note that the results in \cite{Bates2003} are actually strong enough to
handle discretizations of the Laplacian that have infinite range
interactions. In particular, our analysis here should also be applicable
for (regular) dense graphs. We also envision possible extensions to irregularly structured sparse graphs through the use of the recently developed theory of graphops \cite{backhausz2022action,kuehn2020network}.

It is well-known that travelling waves can be used as building blocks to uncover and describe more complex dynamics occurring in spatially extended systems \cite{Aronson1975nonlinear}. We therefore view the current paper as part of a push towards understanding and uncovering the behaviour of systems
in spatially structured environments, which can often naturally be modelled using graphs \cite{arenas2001communication,Selley2015, Slavik2020, Stehlik2017}. The ability to incorporate such spatial structures into models is becoming more and more important in our increasingly networked world.
Indeed, dynamical systems on networks are being
used in an ever-increasing range of disciplines, including chemical reaction theory \cite{feinberg1987chemical},
neuroscience \cite{sporns2016networks}, systems biology \cite{palsson2006systems},
social science \cite{scott1988trend}, epidemiology \cite{barrat2008dynamical}
and transportation
networks \cite{ran2012modeling}.

\paragraph{Organization}
We state our assumptions and main results in \S\ref{sec:mr}
and discuss the general strategy towards solving the
associated fixed point in \S\ref{sec:fix}. The relevant linear
theory is developed in \S\ref{sec:lin}, while the required nonlinear
estimates are obtained in \S\ref{sec:nl}.

\paragraph{Acknowledgments}
Both authors acknowledge support from the Netherlands Organization for Scientific Research (NWO) (grant 639.032.612).

\begin{figure}
\centering
\includegraphics[width=\textwidth]{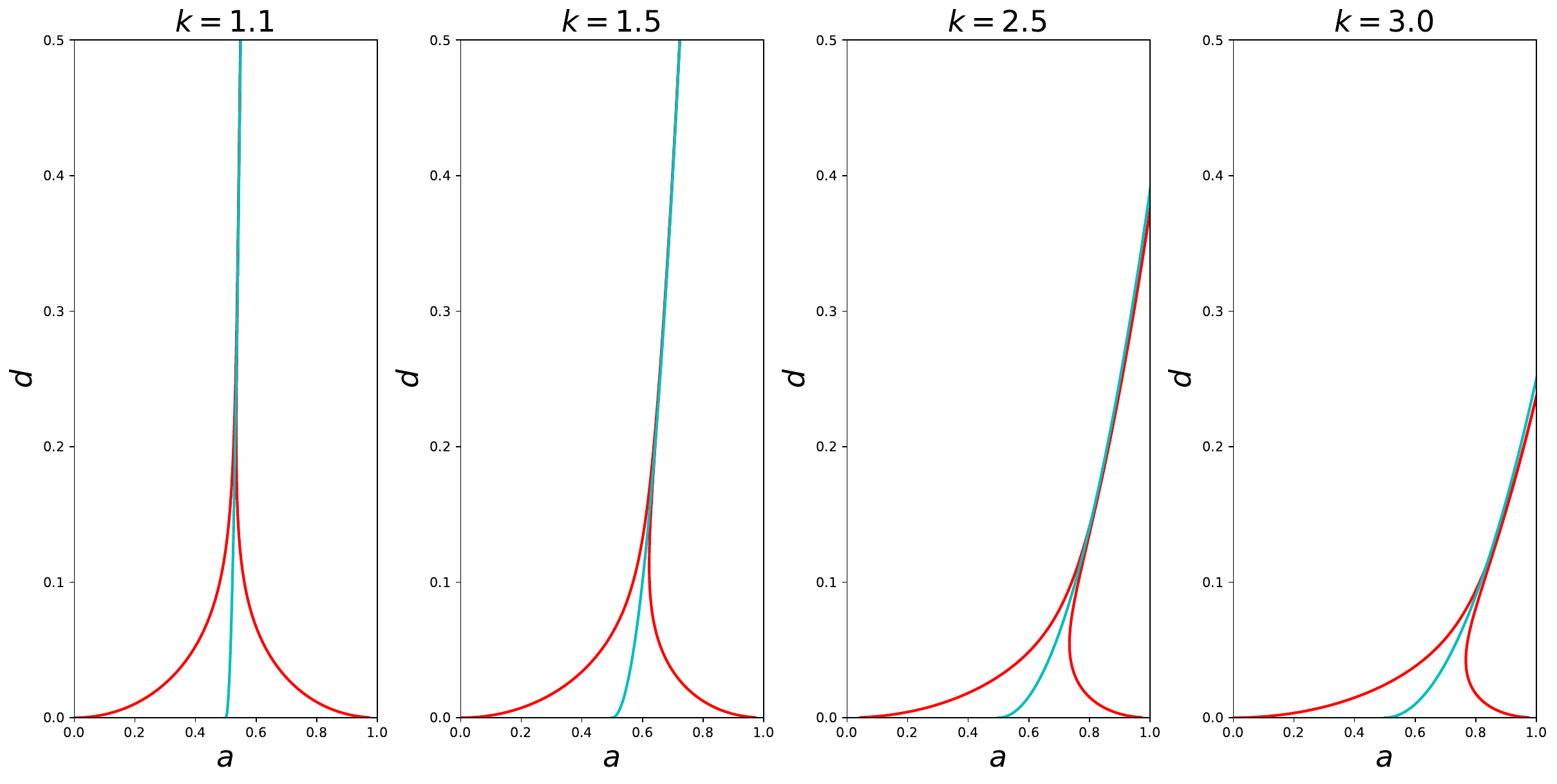}
   \caption{The red curves denote the numerically computed boundaries of the pinning region where $c(a,d,k) = 0$. The cyan
   curves originating from $(1/2, 0)$ represent
   the asymptotic prediction \eqref{eq:int:crit:diff}, which retains its accuracy even for relatively large values of the branching factor $k > 1$.}  
\end{figure} \label{intro:figure}

\section{Main results}
\label{sec:mr}
In order to state our main result for \eqref{eqn:intro:LDE},
we first formalize the bistability condition that we impose
on the nonlinearity $g$.
\begin{itemize}
    \item[(Hg)]
    The map $(u,a)\mapsto g(u;a)$ is $C^1$-smooth  on $\R\times(0,1)$ and 
    for all $a \in (0,1)$ we have the identities
    \begin{align*}
        g(0;a) &= g(a;a) = g(1;a) = 0, 
    \end{align*}
    together with the inequalities
    \begin{align*}
        g'(0;a)  &< 0, \qquad
        g'(1;a) < 0, \qquad
         g'(a;a) > 0
    \end{align*}
    and the sign conditions
    \begin{equation*}
       g(v;a)>0 \ \text{for}\ v\in (-\infty,0) \cup (a,1), \quad g(v;a)<0 \ \text{for}\ v\in (0, a) \cup (1, \infty).
    \end{equation*}
\end{itemize}
Fixing $k > 0$ and turning to travelling waves, it turns out to be convenient to link the diffusion strength $d > 0 $ to a new grid-size parameter $h > 0$
via the (invertible) relations
\begin{equation}
\label{eq:mr:rel:h:d}
    h(d,k) = \dfrac{\sqrt{2}}{\sqrt{d}\sqrt{k+1}},
    \qquad \qquad
    d(h,k) = \dfrac{2}{h^2(k+1)} ,
\end{equation}
which reduce to $d h^2 = 1$ in the symmetric case $k=1$.
To appreciate this, we recall the standard central difference schemes
\eqref{eq;int:central:diff:schemes}
and introduce the associated operators
\begin{equation}
\label{eq:mr:def:disc:derivs}
\begin{array}{lcl}
\delta_h^{0} v (\xi) &:=& \dfrac{v(\xi+ h) - v(\xi - h)}{2h}, \\[0.2cm]
\Delta_h v (\xi) &:= & \dfrac{v(\xi+h) - 2v(\xi) + v(\xi-h)}{h^2}.
\end{array}
\end{equation}
Seeking solutions to \eqref{eqn:intro:LDE} of the form
\begin{equation}\label{eqn:main:ansatz}
    u_i(t) = \Phi\Big(i h(d,k) - ct\Big),
\end{equation}
we find that the (rescaled) profile $\Phi$
must satisfy the mixed-type functional differential equation (MFDE)
\begin{align}\label{eqn:mr:MFDE:with:h}
    -c \Phi'(\xi) = 
    \Delta_h \Phi + 
    \dfrac{2 (k-1)}{h(k+1)} \partial^0_h  \Phi 
     + g(\Phi(\xi);a)
\end{align}
with $h = h(d,k)$,
to which we couple the standard spatial limits
\begin{equation}\label{eqn:mr:BC}
    \lim_{\xi\to-\infty} \Phi(\xi) = 0, \qquad  \lim_{\xi\to\infty} \Phi(\xi) = 1.
\end{equation}

On the one hand, this MFDE is covered by the general framework developed by Mallet-Paret in \cite{ Mallet-Paret1999}. This provides the solutions
that lie at the basis of our analysis in this paper.
\begin{proposition}\label{prop:main:MP}\cite[Thm. 2.1]{Mallet-Paret1999}
Suppose that (Hg) holds and pick 
$a \in (0,1)$ together with $d > 0$ and $k > 0$.
 Then there exist a speed $c=c(a,d, k)$ and a non-decreasing profile $\Phi 
= \Phi(a,d,k):\R\to \R$ that satisfy~\eqref{eqn:mr:MFDE:with:h} with $h = h(d,k)$,
 together with the boundary conditions~\eqref{eqn:mr:BC}.  
 Moreover, $c(a,d, k)$ is uniquely determined and depends $C^1$-smoothly on all parameters when $c(a,d, k) \neq 0$. In this case the profile $\Phi$ is $C^1$-smooth with  $\Phi'>0$ and unique up to translation.
\end{proposition}

On the other hand, under the convergence assumption
\begin{equation}
\label{eq:mr:speed:correction}
  c  + \dfrac{2(k-1)}{h(k+1)} \to \sigma,
\end{equation}
one may readily take 
the formal $h \downarrow 0$ limit of \eqref{eqn:mr:MFDE:with:h}
to arrive at the ODE
\begin{equation}
\label{eq:ode:with:sigma}
  - \sigma \Phi'= \Phi'' + g\big( \Phi ; a \big).
\end{equation}
A classic result (see e.g. \cite{Fife1977}) states that
there is a unique wavespeed $\sigma_{*;a}$ for which
\eqref{eq:ode:with:sigma} with the boundary conditions
\eqref{eqn:mr:BC} admits a solution $\Phi_{*;a}$. This waveprofile is unique up to translation and has $\Phi_{*;a}'> 0$.

The goal of this paper is to link these two viewpoints together from
a spectral convergence perspective, extending earlier
work in \cite{Bates2003} that applies to the $k=1$ case. 
The presence of the discrete derivative $\partial^0_h$ 
in \eqref{eqn:mr:MFDE:with:h} with a
coefficient of size $O(h^{-1})$ introduces complications. 
In fact, our approach can only keep this (large) term under control
if it satisfies a sign condition, which we will achieve
by exploiting the
asymmetry that the parameter $k > 1$ introduces into our problem.
In particular, we need to restrict the values of the detuning parameter $a$
by requiring positive values for the continuum wave speed $\sigma_{*;a}$.
To this end, we introduce the set
\begin{equation}
\begin{array}{lcl}
    \mathcal{A}_*^+ & = & \{ a \in (0,1) : \sigma_{*;a} > 0 \}
    \\[0.2cm]
    & = &  \{ a \in (0,1):  \int_{0}^1 g(s;a) ds < 0 \},
\end{array}
\end{equation}
noting that the second characterization can be obtained by integrating \eqref{eq:ode:with:sigma} against $\Phi'$. With this notation in place,
we are ready to state our main result.

\begin{theorem}[{see {\S}\ref{sec:fix}}]
\label{thm:mr:main}
Assume that $(Hg)$ holds and pick a compact subset $\mathcal{A}_\diamond \subset\mathcal{A}_*^+$. Then there exist constants
$K_{\diamond} > 0$ and $h_{\diamond} > 0$ so that for any $0 < h < h_{\diamond}$, any $a \in\mathcal{A}_\diamond$
and any $k \ge 1$
we have  the bound
\begin{equation}
\label{eq:mr:thm:main:bnd}
    |  c\big( a,d(h,k), k\big) + \dfrac{2(k-1)}{h(k+1)} - \sigma_{*;a} | +   \norm{\Phi\Big(a,d(h,k),k\Big) - \Phi_{*;a}}_{H^1(\R)}  \le K_{\diamond} h.
\end{equation}
\end{theorem}

In order to illustrate the application range of this theorem and highlight the uniformity of
the estimates, we return to the setting of the cubic nonlinearity \eqref{eqn:intro:cubic}.
In this case we have 
\begin{equation}
    \sigma_{*;a} = \sqrt{2}(a - 1/2),
    \qquad \qquad
     \mathcal{A}_*^+ = (\frac{1}{2}, 1)
\end{equation}
and we recall the critical curve
\begin{equation}
 d(a,k) = (a-1/2)^2\dfrac{k+1}{(k-1)^2} 
\end{equation}
discussed in {\S}\ref{sec:int}.
Applying the bound \eqref{eq:mr:thm:main:bnd}
with $h$ as given in \eqref{eq:mr:rel:h:d}, we immediately find
\begin{equation}
\label{eq:mr:err:bnd:c:crit:curve}
    | c(  a,d(a,k), k) | \le \frac{ \sqrt{2} K_\diamond }{ \sqrt{d(a,k) (k+1)} }.
\end{equation}
Our final result formulates a number of such estimates directly in
terms of this critical curve $d$. In particular, we capture the
transition between waves that propagate up the tree
and down the tree. For explicitness,
we have absorbed all unknown constants into a single upper bound $k_* > 1$
for the branch factor. The price is that the exponent $\frac{1}{4}$ 
appearing in the correction curves is not optimal; it can 
be increased up to (but not including) $\frac{1}{2}$.

\begin{corollary}
\label{cor:cubic}
Let $g$ be the standard cubic nonlinearity \eqref{eqn:intro:cubic} and
pick $0 < \delta_a < \frac{1}{4}$. Then there exists $k_* > 1$
so that the following properties hold
for all $a \in [\frac{1}{2} + \delta_a, 1 - \delta_a]$
and $k \in (1, k_*)$.
\begin{itemize}
\item[(i)]{
  We have the bound
  \begin{equation}
  \label{eq:cor:bnd:on:c:crit}
      | c( a, d(a,k), k ) | \le d(a,k)^{-1/4} .
  \end{equation}
}
\item[(ii)]{
  We have $c(a,d,k) < 0$ whenever
  \begin{equation}
      d > d(a,k) \big[ 1 + d(a,k)^{-1/4} ].
  \end{equation}
}
\item[(iii)]{
  We have $c(a,d,k) > 0$ whenever 
  \begin{equation}
      \frac{1}{2} d(a,k) < d < d(a,k) \big[ 1 - d(a,k)^{-1/4}] .
  \end{equation}
}
\end{itemize}
\end{corollary}
\begin{proof}
We write $\mathcal{A}_\diamond = [\frac{1}{2} + \delta_a, 1 - \delta_a]$
and recall the constants $K_\diamond$ and $h_\diamond$ defined in
Theorem \ref{thm:mr:main}. In addition, we write
\begin{equation}
    c_{\mathrm{pred}}(d, a, k) = \sigma_{*;a} - \dfrac{2 (k-1)}{h(d,k)(k+1)}
    = \sqrt{2}(a - \dfrac{1}{2}) - \dfrac{\sqrt{2d}(k-1)}{\sqrt{k+1}} 
\end{equation}
for the implicitly predicted wavespeed appearing
in \eqref{eq:mr:thm:main:bnd}. In order to exploit
these upper and lower bounds,
we pick $\theta \in \{-1, 1\}$, write $d = \nu^2$ and set out to solve
\begin{equation}
\label{eq:cor:pred:bnd}
c_{\mathrm{pred}}( \nu^2,a,k) =\theta K_{\diamond} h(\nu^2,k)
=    \dfrac{\sqrt{2}\theta K_{\diamond}}{\nu\sqrt{k+1}}.
\end{equation}

Upon introducing the expressions
\begin{equation}
\begin{array}{lclcl}
    \gamma(a,k)  &= &
    \dfrac{4 (k-1) K_{\diamond} d(a,k)^{1/4}}{(k+1)(a-1/2)^2} 
    & = & \dfrac{ 4 K_\diamond \sqrt{k-1} }{(k+1)^{3/4} (a-\frac{1}{2})^{3/2}},
    \\[0.3cm]
    \mathcal{E}(a,k,\theta)
    & =  & 
    1-\sqrt{1 - \theta d(a,k)^{-1/4}  \gamma(a,k)}
\end{array}
\end{equation}
the quadratic formula formally yields the two solutions
\begin{equation}
\begin{array}{lcl}
    \nu_-(a,k,\theta) = \frac{1}{2} \sqrt{d(a,k)} \mathcal{E}(a,k,\theta),
    \qquad \qquad
    \nu_+(a,k,\theta) = \frac{1}{2} \sqrt{d(a,k)} \big(2 -  \mathcal{E}(a,k,\theta) \big)
\end{array}
\end{equation}
for \eqref{eq:cor:pred:bnd}.
In addition, we write
\begin{equation}
    \nu_\diamond(k) = \dfrac{\sqrt{2}}{h_\diamond \sqrt{k+1}}
\end{equation}
and observe that $0 < h(d, k) < h_\diamond$ whenever $d > \nu_\diamond(k)^2$.

We now choose $k_* >1 $ in such a way that for all $k \in (1, k_*)$
and $a \in \mathcal{A}_\diamond$
we have
\begin{equation}
    d(a,k) \ge 1,  
    \qquad \qquad \gamma(a,k) \le 1,
\end{equation}
ensuring that the quantities above are all real-valued.
Upon writing
\begin{equation}
\begin{array}{lcl}
    \mathcal{D}_-(a,k) & = & \Big( \max\{\nu_\diamond(k), \nu_+(a,k,-1) \}^2,  \infty \Big), \\[0.2cm]
    \mathcal{D}_+(a,k) & = & \Big( \max\{ \nu_\diamond(k), \nu_-(a,k,+1)\}^2 ,  \nu_+(a,k,+1)^2 \Big), \\[0.2cm]
\end{array}
\end{equation}
it follows from  \eqref{eq:mr:thm:main:bnd}
that $c(a,d,k) < 0$ for $d \in \mathcal{D}_-(a,k)$
and $c(a,d,k) > 0$ for $d \in \mathcal{D}_+(a,k)$.

Using the bound
$\sqrt{1 + x} \le 1 + \frac{1}{2} x$ for $x \ge 0$, we obtain
\begin{equation}
    -\mathcal{E}(a,k, -1) \le \frac{1}{2} d(a,k)^{-1/4} \gamma(a,k)
    \le \frac{1}{2} d(a,k)^{-1/4}
\end{equation}
and hence
\begin{equation}
    \nu_+(a,k,-1)^2 \le \frac{1}{4} d(a,k)\big(4+2 d(a,k)^{-1/4}+ \frac{1}{4} d(a,k)^{-1/2} \big) \le d(a,k) \big( 1 + d(a,k)^{-1/4} \big).
\end{equation}
On the other hand,
the bound $1 - \sqrt{1-y} \le y$ for $0 \le y \le 1$
implies that
\begin{equation}
    \mathcal{E}(a,k, +1) \le d(a,k)^{-1/4} \gamma(a,k) \le d(a,k)^{-1/4}
\end{equation}
and hence
\begin{equation}
\begin{array}{lclcl}
    \nu_-(a,k,+1)^2 & \le & \frac{1}{4} d(a,k) d(a,k)^{-1/2} 
     & \le & \frac{1}{4} d(a,k) ,
    \\[0.2cm]
    \nu_+(a,k,+1)^2 & \ge & \frac{1}{4} d(a,k)\big(4 - 4d(a,k)^{-1/4} +d(a,k)^{-1/2} \big) 
    &\ge & d(a,k)\big(1 - d(a,k)^{-1/4} \big) .
\end{array}
\end{equation}
By further restricting $k_* > 1$ we can ensure that $d(a,k) > 2\nu_\diamond(k)^2$ and that \eqref{eq:mr:err:bnd:c:crit:curve}
can be simplified to \eqref{eq:cor:bnd:on:c:crit}, which
completes the proof.
\end{proof}

\section{Fixed point problem}
\label{sec:fix}

In this section we setup the fixed point problem that 
will enable us to extract the bounds \eqref{eq:mr:thm:main:bnd}.
In particular, we isolate the correct linear and nonlinear parts
and formulate a convergence result for the associated linear operators.
The overall strategy closely resembles the approach originally developed in \cite{Bates2003} and generalized
in \cite{hupkesbtravelling,HJHFHNINFRANGE}.

In order to extract the anticipated speed correction
\eqref{eq:mr:speed:correction}, we recall
the discrete derivatives \eqref{eq:mr:def:disc:derivs}
and introduce the combined operator
\begin{equation}
\label{eq:fxp:def:M:h:k}
    M_{h,k}v = \Delta_h v + \dfrac{2(k-1)}{h (k+1)} \big[ \partial_h^0 v - v' \big].
\end{equation}
This allows us to recast the MFDE \eqref{eqn:mr:MFDE:with:h} 
in the form
\begin{equation}
\label{eq:mfde:with:M}
    -\big[ c + \dfrac{2(k-1)}{h (k+1)} \big] \Phi' = M_{h,k} \Phi + g( \Phi;a).
\end{equation}
We now consider the pair $(c,\Phi)$ as a perturbation
from the profile $\Phi_{*;a}$ and the
anticipated wavespeed \eqref{eq:mr:speed:correction}
by writing
\begin{equation}
\label{eq:fxp:pert:phi:c}
    \Phi = \Phi_{*;a} + v,
    \qquad \qquad
    c =\sigma_{*;a}  - \frac{2(k-1)}{h(k+1)} +  \overline{c}
\end{equation}
A direct computation shows that solving \eqref{eq:mfde:with:M} is equivalent to finding a solution $(\overline{c}, v)$ to the problem
\begin{equation}\label{eqn:fxp:reform:LR}
    \mathcal{L}_{h,  k;a} v= \overline{c} \Phi_{*;a}'
    + \mathcal{R}_A(\overline{c}, v;a) + 
    \mathcal{R}_B(h,k;a),
\end{equation}
where we have introduced the two nonlinearities
\begin{equation}
\begin{array}{lcl}
    \mathcal{R}_A(\overline{c}, v;a) &=& 
     \overline{c} v'  + g(\Phi_{*;a}+v;a) -  g(\Phi_{*;a};a) - g_u(\Phi_{*;a};a)v,  \\[0.2cm]
     \mathcal{R}_B(h,k;a) & = & 
         M_{h,k} \Phi_{*;a} - \Phi_{*;a}''
\end{array}
\end{equation}
together with the linear operator
\begin{equation}
\mathcal{L}_{h,  k;a} v = - \sigma_{*;a} v' -  M_{h,k} v  - g'(\Phi_{*;a};a)v  .
\end{equation}

The key point is that the nonlocal operator $\mathcal{L}_{h, k;a}$ formally converges to the well-known second order differential operator $\mathcal{L}_{*;a}: H^2 \to L^2$ that acts as
\begin{equation}\label{eqn:limit:L0}
  [\mathcal{L}_{*;a} v](\xi) = - c_{*;a} v'(\xi)- v''(\xi) - g'(\Phi_{*;a}(\xi);a)v(\xi).
\end{equation}
This operator is Fredholm \cite{Bates2003} and satisfies
\begin{equation}
\label{eq:fxp:ker:range:l:star}
\mathrm{Ker} \, \mathcal{L}_{*;a}
= \mathrm{span}\{ \Phi_{*;a}' \},
\qquad \qquad 
\mathrm{Range} \, \mathcal{L}_{*;a} = \{ f : \langle \Psi_{*;a}, f \rangle_{L^2} = 0 \},
\end{equation}
where we introduced the adjoint eigenfunction
\begin{equation}
    \Psi_{*;a}(\xi) :=  \Phi_{*;a}'(\xi) e^{-\sigma_{*;a} \xi} / \int \Phi_{*;a}'(\xi')^2 e^{-\sigma_{*;a} \xi'} \, d \xi'
\end{equation}
normalized to have
\begin{equation}
    \langle \Psi_{*;a}, \Phi_{*;a}' \rangle_{L^2} = 1.
\end{equation}
Our first main contribution is the analogue of \cite[Thm. 2.3]{hupkesbtravelling} for the current setting and shows
in a sense that the characterization \eqref{eq:fxp:ker:range:l:star} can be transferred to the operators $\mathcal{L}_{h,k;a}$.

\begin{proposition}[{see {\S}\ref{sec:lin}}]
\label{prop:bates:inverseWithNormalz}
Suppose that $(Hg)$ holds and pick a compact set
$\mathcal{A}_\diamond \subset \mathcal{A}_*^+$.
Then there exist constants $K > 0$
and $h_0 > 0$ together with linear maps
\begin{equation}
\beta_{h,k;a}: L^2 \to \mathbb{R},
\qquad \qquad
\mathcal{V}_{h,k;a}: L^2 \to H^1,
\end{equation}
defined for all $h \in (0, h_0)$, $k \ge 1$
and $a \in \mathcal{A}_\diamond$,
so that the following properties hold true
for all such $(h,k,a)$.
\begin{itemize}
\item[(i)]{
  For all $f \in L^2$ 
  we have the bound
  \begin{equation}
    \abs{ \beta_{h,k;a} f }
      + \norm{ \mathcal{V}_{h,k;a} f }_{H^1}
           \le K \norm{f}_{L^2}.
  \end{equation}
}
\item[(ii)]{
  For all $f \in L^2$, 
  the pair
  \begin{equation}
    (\beta, v) = \big( \beta_{h,k;a} f, \mathcal{V}_{h,k;a} f \big) \in
    \mathbb{R} \times H^1
  \end{equation}
  is the unique solution to the problem
  \begin{equation}
    \label{eq:bates:qinv:problem}
    \mathcal{L}_{h,k;a} v = f + \beta \Phi_{*;a}'
  \end{equation}
  that satisfies the normalization condition
  \begin{equation}
    \label{eq:bates:qinv:problem:norm}
    \langle \Psi_{*;a} , v \rangle_{L^2} = 0.
  \end{equation}
}
\item[(iii)]{
  We have $\beta_{h,k;a} \Phi_*' = -1$. 
}
\end{itemize}
\end{proposition}

In view of \eqref{eq:bates:qinv:problem},
the problem \eqref{eqn:fxp:reform:LR}
can now be rewritten in the fixed-point form
\begin{equation}
\label{eq:fxp:setup:fix:pnt:prob}
[\overline{c}, v] = [ \beta^*_{h,k;a}, \mathcal{V}^*_{h,k;a}] \big( \mathcal{R}_A(\overline{c}, v) + \mathcal{R}_B(h,k)\big).
\end{equation}
Our second result here constructs solutions
to this problem in the set
\begin{equation}
\mathcal{Z}_\mu = \{ (\overline{c}, v) \in \mathbb{R} \times H^1: 
|\overline{c}| + \norm{v}_{H^1} \le \mu
\},
\end{equation}
which automatically accounts for the boundary conditions
\eqref{eqn:mr:BC}.

\begin{proposition}[{see {\S}\ref{sec:nl}}]
\label{prop:fxp:reformulation:exist}
Consider the setting of Theorem \ref{thm:mr:main}. Then there exists $h_\diamond>0$ 
and $K_\diamond > 0$ such that for all $h\in (0, h_\diamond)$, $k \ge 1$ and $a \in \mathcal{A}_\diamond$, the fixed point problem 
\eqref{eq:fxp:setup:fix:pnt:prob} posed on the set $\mathcal{Z}_{\mu}$ with 
$\mu = K_\diamond h$
has a unique solution $(\overline{c}, v)$.
\end{proposition}

\begin{proof}[Proof of Theorem \ref{thm:mr:main}] The result follows  directly from Proposition \ref{prop:fxp:reformulation:exist} and the identifications \eqref{eq:fxp:pert:phi:c}.
\end{proof}


\section{Linear theory }
\label{sec:lin}

Our goal in this section is to establish Proposition \ref{prop:bates:inverseWithNormalz}, modifying
the approach in \cite{Bates2003} to 
account for the $O(h^{-1})$ term present in the operator 
$M_{h,k}$ defined in \eqref{eq:fxp:def:M:h:k}. As a preparation,
we introduce the formal adjoints
\begin{equation}
    M_{h,k}^{\mathrm{adj}} w = \Delta_h w - 
    \dfrac{2(k-1)}{h(k+1)} \big[ \partial^0_h w - w'\big]
\end{equation}
together with
\begin{equation}\label{eqn:reformulation:adjoint}
    \mathcal{L}_{h, k;a}^{\mathrm{adj}}w = 
    \sigma_{*;a} w' - M_{h,k}^{\mathrm{adj}} w   - g'(\Phi_{*;a};a) w .
\end{equation}
Our key task is to establish lower bounds for the quantities
\begin{equation}
\label{eq:linop:defCalEkm}
\begin{array}{lcl}
\mathcal{E}_{ h}(\delta)  & = &
  \inf_{a \in \mathcal{A}_\diamond, k \ge 1}
  \inf_{ \norm{v}_{H^1} = 1 }
\big\{  \norm{\mathcal{L}_{h,k;a} v - \delta v}_{ L^2 }
 + \delta^{-1} \abs{
   \langle  \Psi_{*;a},
                        \mathcal{L}_{h,k;a} v - \delta v
             \rangle_{ L^2 }   }
             \big\},
\\[0.2cm]
\mathcal{E}_{h}^{\mathrm{adj}}(\delta)  & = &
  \inf_{a \in \mathcal{A}_\diamond,k \ge 1} \inf_{ \norm{w}_{H^1} = 1 }
\big\{  \norm{\mathcal{L}^{\mathrm{adj}}_{h,k;a} w - \delta w}_{L^2 }
 + \delta^{-1} \abs{
   \langle \Phi_{*;a}',
             \mathcal{L}^{\mathrm{adj}}_{h,k;a} w - \delta w
             \rangle_{ L^2 }   }  \big\}
\end{array}
\end{equation}
as stated in the following result,
which is analogous to \cite[Lem. 6]{Bates2003}.

\begin{proposition}
\label{prp:bates:lowBoundOnE}
Suppose that (Hg) is satisfied and pick 
a compact set $\mathcal{A}_\diamond \subset \mathcal{A}_0^+$.
Then there exists $\mu > 0$ and $\delta_0 > 0$
such that for every $0 < \delta < \delta_0$
we have
\begin{equation}
\label{eq:bates:unifLowBoundKappa}
\begin{array}{lcl}
\mu(\delta)   &: = &
\mathrm{liminf}_{h \downarrow 0  } \,
  \mathcal{E}_{h}(\delta) \ge \mu,
\\[0.2cm]
\mu^{\mathrm{adj}}(\delta) &: = &
  \mathrm{liminf}_{h \downarrow 0 } \,
   \mathcal{E}^{\mathrm{adj}}_{h}(\delta) \ge \mu.
\end{array}
\end{equation}
\end{proposition}

Indeed, for small $\delta > 0$ these lower bounds readily allow us to  extend
the estimate
\begin{equation}
\label{eq:prp:bates:ode:inv}
\norm{ (\mathcal{L}_{*;a} - \delta)^{-1} f}_{H^2}
\le K\Big[ \norm{f}_{L^2}
+ \delta^{-1} \big| \langle  \Psi_{*;a} ,  f  \rangle_{L^2} \big|
\Big] 
\end{equation}
that is available for the limiting second-order differential operator
\eqref{eqn:limit:L0}; see e.g. \cite[Lem. 5]{Bates2003}.
We remark that the constant $K$ in \eqref{eq:prp:bates:ode:inv} can be chosen uniformly 
with regards to the parameter $a$ in compact subsets of $(0,1)$ on account of the smoothness of $g$.
This extension result should be compared
to \cite[Thm. 4]{Bates2003} and can be used
to establish Proposition \ref{prop:bates:inverseWithNormalz}
by following the procedure in \cite{hupkesbtravelling}.

\begin{corollary}
\label{prp:bates:invertibleBounds}
Suppose that (Hg) is satisfied and pick a
compact set $\mathcal{A}_\diamond \subset \mathcal{A}_0^+$.
There exists constants $K > 0$ and $\delta_0 > 0$
together with a map $h_0: (0, \delta_0) \to (0, 1)$
so that the following holds true.
For any $0 < \delta < \delta_0$,
any $0 < h < h_0(\delta)$, any $k \ge 1$ and any $a \in \mathcal{A}_\diamond$, the operator
$\mathcal{L}_{h,k;a} - \delta$ is invertible as a map
from $H^1$ onto $L^2$ and satisfies the bound
\begin{equation}
\label{eq:prp:bates:inv:bounds:unif:bnd:lmind:inv}
\norm{ (\mathcal{L}_{h,k;a} - \delta)^{-1} f}_{H^1}
\le K\Big[ \norm{f}_{L^2}
+ \delta^{-1} \big| \langle  \Psi_{*;a} ,  f  \rangle \big|
\Big] .
\end{equation}
\end{corollary}
\begin{proof} 
Following the proof of
\cite[Thm. 4]{Bates2003},
we fix $0 < \delta < \delta_0$  and a
sufficiently small $h > 0$. We subsequently pick
an arbitrary $k \ge 1$ and $a \in \mathcal{A}_\diamond$.
By Proposition \ref{prp:bates:lowBoundOnE},
the operator $\mathcal{L}_{h,k;a} - \delta$ is an homeomorphism
from $H^1$ onto its range
\begin{equation}
  \mathcal{R} = (\mathcal{L}_{h,k;a} - \delta) \big( H^1 \big) \subset L^2,
\end{equation}
with a bounded inverse $\mathcal{I}: \mathcal{R} \to H^1$. The latter fact
shows that $\mathcal{R}$ is a closed subset of $L^2$.
If $\mathcal{R} \neq L^2$, there exists a non-zero
$w \in L^2$ so that $\langle w, \mathcal{R} \rangle_{L^2} = 0$,
i.e.,
\begin{equation}
\big\langle w , (\mathcal{L}_{h,k;a} - \delta) v \big\rangle_{L^2} = 0 \hbox{ for all }
  v \in H^1.
\end{equation}
Restricting this identity to test functions $v \in C_c^\infty$
implies that in fact $w \in H^1$.
In particular, we find
\begin{equation}
\big\langle ( \mathcal{L}^{\mathrm{adj}}_{h,k;a} - \delta) w , v
  \big\rangle_{L^2} = 0 \hbox{ for all } v \in H^1,
\end{equation}
which by the density of $H^1$ in $L^2$ means that $(\mathcal{L}^{\mathrm{adj}}_{h,k;a} - \delta) w = 0$.
Applying Proposition \ref{prp:bates:lowBoundOnE} once more
yields the contradiction $w = 0$
and establishes $\mathcal{R} = L^2$.
The bound \eqref{eq:prp:bates:inv:bounds:unif:bnd:lmind:inv}
with the constant $K > 0$ that does not depend on the parameters $(\delta, k,a)$
now follows directly from the definition \eqref{eq:linop:defCalEkm} of the
quantities $\mathcal{E}_{h}(\delta)$
and the uniform lower bound \eqref{eq:bates:unifLowBoundKappa}.
\end{proof}

\begin{proof}[Proof of Proposition \ref{prop:bates:inverseWithNormalz}]
In view of the uniform bound
\eqref{eq:prp:bates:inv:bounds:unif:bnd:lmind:inv},
the computations in the proof of \cite[Thm. 2.3]{hupkesbtravelling}
can be followed (almost) verbatim.
\end{proof}

\subsection{Proof of Proposition \ref{prp:bates:lowBoundOnE}}

Setting out to find lower bounds for the quantities
\eqref{eq:linop:defCalEkm},
we first provide some basic properties of the operator $M_{h,k}$.
An important difference with \cite[Lem. 3]{Bates2003} is
that the estimate \eqref{eq:m:h:k:ineq:v:vp}
only provides inequalities instead of the equalities that are possible for $\Delta_h$.

\begin{lemma}
Consider a function $f \in C^3$, write
\begin{equation}
   f_{3;\infty} := \xi \mapsto \sup_{|\xi'- \xi| \le 1} |f'''(\xi')| \in L^2
\end{equation}
and suppose that $f_{3,\infty} \in L^2$. Then
for any $0 < h < 1$ and $k \ge 1$ 
we have the bound
\begin{equation}
\label{eq:lim:m:h:k:unif:c3}
 \norm{ M_{h,k} f - f''}_{L^2} 
+ \norm{ M_{h,k}^{\mathrm{adj}} f - f''}_{L^2} 
\le 2 h \norm{f_{3,\infty}}_{L^2}.
\end{equation}
In addition,  the inequalities
\begin{equation}
\label{eq:m:h:k:ineq:v:v}
    \langle M_{h,k} v, v \rangle_{L^2} \le 0,
    \qquad \qquad
    \langle M_{h,k}^{\mathrm{adj}} v, v \rangle_{L^2} \le 0
\end{equation}
hold for any $h > 0$, $k \ge 1$ and $v \in L^2$.
If in fact $v \in H^1$, then we also have the inequalities
\begin{equation}
\label{eq:m:h:k:ineq:v:vp}
    \langle M_{h,k} v, v'\rangle_{L^2} \le 0,
    \qquad \qquad
    \langle M_{h,k}^{\mathrm{adj}} v, v'\rangle_{L^2} \ge 0.
\end{equation}
\end{lemma}
\begin{proof}
In view of the uniform bound $0 \le (k-1)/(k+1) \le 1$ for $k \ge 1$,
the estimate \eqref{eq:lim:m:h:k:unif:c3} follows 
directly from a Taylor expansion. Turning to 
the remaining inequalities,
we note that the Fourier symbol associated to $M_{h,k}$ is given by
\begin{equation}
\begin{array}{lcl}
    \widehat{M}_{h,k}(\omega)
    & = & 
    h^{-2} \Big[ 2 (\cos (\omega h) -1) +  \dfrac{2i (k-1)}{k+1}(\sin (\omega h) - \omega h) \Big] .
    \end{array}
\end{equation}
In particular, the functions
\begin{equation}
    \omega \mapsto \mathrm{Re} \, \widehat{M}_{h,k}(\omega),
    \qquad \qquad
    \omega \mapsto  \omega \, \mathrm{Im} \, \widehat{M}_{h,k}(\omega)
\end{equation}
are both even and non-positive, from which the claims follow
readily.
\end{proof}

The next step is to show that
the limiting
values in \eqref{eq:bates:unifLowBoundKappa}
can be approached
via a sequence of realizations
that convergence in an appropriate
weak sense. The key point is that weak limits
can also be extracted from the operators $M_{h,k}$
when considered on appropriate sequences in $H^1$;
see \eqref{eq:lin:seq:lims:m:vw}.

\begin{lemma}
\label{lem:bates:approxSetting}
Consider the setting of Proposition
\ref{prp:bates:lowBoundOnE} and fix a constant
$\delta > 0$.
Then there exist triplets 
\begin{equation}
(a_*, V_*, Y_* ) \in \mathcal{A}_\diamond \times H^2 \times L^2,
\qquad
(\tilde{a}_*, W_*, Z_*) \in \mathcal{A}_\diamond \times H^2 \times L^2,
\end{equation}
together with sequences
\begin{equation}
\begin{array}{lcl}
\label{eq:lem:linop:defSeqs}
\{(h_j, a_j, k_j, v_j,  y_j ) \}_{j \in \mathbb{N}}
&\subset&
 (0, 1) \times \mathcal{A}_\diamond \times [1,\infty) \times H^1 \times L^2, 
\\[0.2cm]
\{(\tilde{h}_j, \tilde{a}_j, \tilde{k}_j, w_j, z_j ) \}_{j \in \mathbb{N}}
&\subset&
 (0, 1) \times \mathcal{A}_\diamond \times [1,\infty) \times H^1 \times L^2 
\end{array}
\end{equation}
%
that satisfy the following properties.
\begin{itemize}
\item[(i)]{
  For any $j \in \mathbb{N}$, we have
  \begin{equation}
    \norm{v_j}_{H^1} = \norm{w_j}_{H^1} = 1 ,
  \end{equation}
  together with
  \begin{equation}
  \label{eq:lem:seq:id:for:v:w}
  \begin{array}{lcl}
  \mathcal{L}_{ h_j,k_j;a_j} [v_j ] - \delta v_j & = & y_j ,
  \\[0.2cm]
  \mathcal{L}^{\mathrm{adj}}_{ \tilde{h}_j,\tilde{k}_j;\tilde{a}_j} [w_j ] - \delta w_j & = & z_j .
  \end{array}
  \end{equation}
}
\item[(ii)]{
  Recalling the constants 
  defined in
  \eqref{eq:bates:unifLowBoundKappa},
  we have $h_j \downarrow 0$ and $\tilde{h}_j \downarrow 0$ together with the limits
  \begin{equation}
  \label{eq:lem:lowBndInv:limit}
    \begin{array}{lcl}
    \mu(\delta) & = &
       \lim_{j \to \infty} \{ \norm{y_j}_{L^2}
     + \delta^{-1} \big| \langle \Psi_{*;a_j},  y_j 
          \rangle_{ L^2 } \big| \} ,
    \\[0.2cm]
    \mu^{\mathrm{adj}}(\delta) & = &
       \lim_{j \to \infty} \{ \norm{z_j}_{L^2}
     + \delta^{-1} \big| \langle \Phi_{*;\tilde{a}_j}',  z_j
          \rangle_{ L^2 } \big| \} .
    \\[0.2cm]
    \end{array}
  \end{equation}
}
\item[(iii)]{
 As $j \to \infty$, we have $a_j \to a_*$
 and $\tilde{a}_j \to \tilde{a}_*$. In addition, we have the weak convergences
 \begin{equation}
   \label{eq:lin:seq:lims:vw}
   v_j \rightharpoonup V_* \in H^1,
   \qquad
   w_j \rightharpoonup W_* \in H^1,
 \end{equation}
 together with
 \begin{equation}
   \label{eq:lin:seq:lims:m:vw}
   M_{h_j,k_j} v_j \rightharpoonup V_*'' \in L^2,
   \qquad
   M_{\tilde{h}_j,\tilde{k}_j}^{\mathrm{adj}} w_j \rightharpoonup W_*'' \in L^2,
 \end{equation}
 and finally
 \begin{equation}
  \label{eq:lin:seq:lims:yz}
   y_j \rightharpoonup Y_* \in L^2,
   \qquad
   z_j \rightharpoonup Z_* \in L^2.
 \end{equation}
}
\end{itemize}
\end{lemma}
\begin{proof}
The existence of the sequences
\eqref{eq:lem:linop:defSeqs}
that satisfy (i) and (ii)
with $h_j \downarrow 0$ and $\tilde{h}_j \downarrow 0$
follows directly from the definitions
\eqref{eq:bates:unifLowBoundKappa}.
Notice that \eqref{eq:lem:lowBndInv:limit}
implies that we can pick $C' > 0$ for which
we have the uniform bound
\begin{equation}
\norm{y_j}_{L^2}
 + \norm{z_j}_{L^2} \le C'
\end{equation}
for all $j \in \mathbb{N}$.
In particular, after taking a subspace
we obtain \eqref{eq:lin:seq:lims:vw}
and \eqref{eq:lin:seq:lims:yz}. 
To obtain \eqref{eq:lin:seq:lims:vw}, we notice that
$M_{h_j,k_j} v_j$ and $M_{h_j,k_j}^{\mathrm{adj}} w_j$ are bounded sequences in $L^2$ since $y_j$ and $z_j$ are. In particular, there exist
$(V_*^{(2)}, W_*^{(2)}) \in L^2 \times L^2$ for which we have
the weak limits
\begin{equation}
    M_{h_j,k_j} v_j\rightharpoonup V_*^{(2)},
    \qquad \qquad
    M^{\mathrm{adj}}_{h_j,k_j} v_j\rightharpoonup W_*^{(2)}.
\end{equation}
Focusing on the first sequence, 
we now pick an arbitrary test function 
    $\zeta \in C_c^\infty$ and compute
    \begin{equation}
    \begin{array}{lcl}
        \langle M_{h_j,k_j} v_j, \zeta \rangle_{L^2} 
        &=&   \langle v_j, \zeta''\rangle_{L^2}
        + \langle v_j,  M_{h_j,k_j}^{\mathrm{adj}} \zeta - \zeta'' \rangle_{L^2}
        \\[0.2cm]
        & = & - \langle v_j', \zeta' \rangle_{L^2}
        + \langle v_j,  M_{h_j,k_j}^{\mathrm{adj}} \zeta - \zeta'' \rangle_{L^2} .
    \end{array}
    \end{equation}
    Taking limits, and using the uniform convergence \eqref{eq:lim:m:h:k:unif:c3}  we hence find
    \begin{equation}
      \langle V_*^{(2)}, \zeta \rangle_{L^2}
      = - \langle V_j', \zeta' \rangle_{L^2},
    \end{equation}
    which by the density of test functions in $L^2$ implies
    that $V_* \in H^2$ with $V_*'' = V_*^{(2)}$. The analogous argument
    works for $W_*$.
\end{proof}

In the remainder of this section we 
obtain upper and lower bounds for the size of
the limiting functions $V_*$ and $W_*$. 
Upper bounds can be obtained relatively
directly from 
\eqref{eq:prp:bates:ode:inv}
following the procedure in  \cite[{\S}3.2]{Bates2003}.

\begin{lemma}
\label{lem:bates:approx:upBnd}
Consider the setting
of Proposition \ref{prp:bates:lowBoundOnE}.
There exist constants $K_1 > 0$ and $\delta_0 > 0$ so
that for any $0 < \delta < \delta_0$,
the functions $V_*$ and $W_*$ defined in Lemma
\ref{lem:bates:approxSetting}
satisfy the bounds 
\begin{equation}
\norm{V_*}_{H^2 } \le K_1 \mu(\delta),
\qquad \qquad
\norm{W_*}_{H^2 } \le K_1 \mu^{\mathrm{adj}}(\delta).
\end{equation}
\end{lemma}
\begin{proof}
Using item (iii) of Lemma \ref{lem:bates:approxSetting}
to take the weak limit of \eqref{eq:lem:seq:id:for:v:w},
we find that
\begin{equation}
\big[\mathcal{L}_{*;a_*} - \delta \big][V_*] = Y_* .
\end{equation}
The lower-semicontinuity of the $L^2$-norm
under weak limits implies that
\begin{equation}
\norm{Y_*}_{L^2} + \delta^{-1}
  \big| \langle \Psi_{*;a_*} ,  Y_* \rangle_{L^2} \big|
     \le \mu(\delta),
\end{equation}
so the conclusion follows from
the uniform estimate \eqref{eq:prp:bates:ode:inv}.
The bound for $W_*$ follows in a similar fashion.
\end{proof}

The next result controls the size of the derivatives
$(v_j', w_j')$, which is crucial to rule out
the leaking of energy into oscillations
that are not captured by the relevant weak limits.
It is here that we need to use the inequalities
\eqref{eq:m:h:k:ineq:v:vp} instead of the usual equalities.
This requires us to impose the restriction  $\sigma_{*;a} > 0$, 
corresponding to the fact that the reflection symmetry 
breaks when passing from a grid to a tree.

\begin{lemma}
\label{lem:bates:approx:bnd:on:deriv}
Consider the setting
of Proposition \ref{prp:bates:lowBoundOnE}.
There exists a constant $K_2 > 0$ that does not depend on $\delta > 0$ 
so that the sequences in Lemma \ref{lem:bates:approxSetting} satisfy the inequalities
\begin{equation}
\label{eq:bates:deriv:up:bnds}
\begin{array}{lcl}
\norm{v_j'}_{L_{2}}^2 & \le &
 K_2 \Big[
   \norm{y_j}_{L^2}^2
   + \norm{v_j}_{L^2}^2
 \Big] ,
\\[0.2cm]
\norm{w_j'}_{L_{2}}^2 & \le &
 K_2 \Big[
   \norm{z_j}_{L^2}^2
   + \norm{w_j}_{L^2}^2
 \Big]
\end{array}
\end{equation}
for all $j > 0$.
\end{lemma}
\begin{proof}
We expand the identity
\begin{equation}
\langle \mathcal{L}_{h_j,k_j;a_j} v_j - \delta v_j , v_j'
\rangle_{L^2} = \langle y_j,  v_j' \rangle_{L^2}
\end{equation}
to obtain
\begin{equation}
\begin{array}{lcl}
\sigma_{*;a_j} \langle  v_j', v_j' \rangle_{L^2}
+ \langle y_j,  v_j' \rangle_{L^2}
& = &
 - \delta \langle v_j, v_j' \rangle_{L^2}
+ \langle  M_{h_j,k_j} v_j, v_j' \rangle_{L^2}
+ \langle  g'(\Phi_{*;a_j};a_j) v_j, v_j' \rangle_{L^2}.
\\[0.2cm]
\end{array}
\end{equation}
Using the identity
$\langle  v_j, v_j' \rangle_{L^2} = 0$
and the inequality \eqref{eq:m:h:k:ineq:v:vp},
we may hence compute
\begin{equation}
\label{eq:linop:bnd:vpj:vpj}
\begin{array}{lcl}
\sigma_{*;a_j} \langle v'_j, v'_j \rangle_{L^2}
 & \le & C'\Big[
  \norm{v_j}_{L^2} \norm{v_j'}_{L^2}
  + \norm{y_j}_{L^2} \norm{v_j'}_{L^2}
  \Big] 
\end{array}
\end{equation}
for some constant $C'>0$.
We now use the compactness of $A_\diamond$
to obtain a strictly positive lower bound 
for $\sigma_{*;a_j}$. Dividing
\eqref{eq:linop:bnd:vpj:vpj}
through by
$\norm{v_j'}_{L^2}$
and squaring, we hence obtain
\begin{equation}
\begin{array}{lcl}
\norm{v_j'}_{L_{2}}^2 \le
 K_2 \Big[
    \norm{v_j}_{L^2}^2
   + \norm{y_j}_{L^2}^2
   \Big] 
\end{array}
\end{equation}
for some $K_2 > 0$. The same procedure works for $w_j'$.
\end{proof}

We are now ready to
obtain lower bounds for $\norm{V_*}_{H^1}$
and $\norm{W_*}_{H^1}$.
Arguing as in \cite{Bates2003}, the key
ingredient is the bistable nature of our nonlinearity.
Indeed, this allows us to restrict attention to a
compact interval on which (subsequences of) the series $v_j$ and $w_j$ converge strongly.

\begin{lemma}
\label{lem:bates:approx:lowBnd}
Consider the setting
of Proposition \ref{prp:bates:lowBoundOnE}.
There exists constants $K_3 > 0$
 and $K_4 > 0$  so
that for any $0 < \delta < \delta_0$
the functions $V_*$ and $W_*$ defined in Lemma
\ref{lem:bates:approxSetting}
satisfy the bounds
\begin{equation}
\begin{array}{lcl}
\norm{V_*}^2_{H^1 }
  & \ge &  K_3 - K_4 \mu(\delta)^2,
\\[0.2cm]
\norm{W_*}^2_{H^1 }
  & \ge &  K_3 - K_4 \mu^{\mathrm{adj}}(\delta)^2.
\end{array}
\end{equation}
\end{lemma}
\begin{proof}
Pick $m > 1$ and $\alpha > 0$
in such a way that
\begin{equation}
 g'\big(\Psi_{*;a_j}(\tau);a_j\big) \le - \alpha
\end{equation}
holds for all $\abs{\tau} \ge m$ and all $j$.
This is possible on account
of the fact that
$g'(0;a_j) < 0$ and $g'(1;a_j) < 0$, the compactness of $\mathcal{A}_\diamond$
and the smoothness of $g$.

We now expand the identity
\begin{equation}
\langle \mathcal{L}_{h_j,k_j;a_j} v_j - \delta v_j ,
v_j
\rangle_{L^2} = \langle y_j,
  v_j \rangle_{L^2}
\end{equation}
to obtain the estimate
\begin{equation}
\begin{array}{lcl}
 \langle y_j, v_j \rangle_{ L^2 }
& = & - \sigma_{*;a} \langle v_j', v_j \rangle_{L^2 }
   - \delta \langle v_j, v_j \rangle_{L^2}
\\[0.2cm]
& & \qquad
+   \langle 
  M_{h_j,k_j} v_j, v_j \rangle_{L^2}
+ \langle  g'(\Psi_{*;a_j};a_j) v_j, v_j \rangle_{L^2}
\\[0.2cm]
\end{array}
\end{equation}
Using $\langle v_j', v_j \rangle_{L^2} = 0$,
and the inequality \eqref{eq:m:h:k:ineq:v:v}
we find
\begin{equation}
\begin{array}{lcl}
 \langle y_j, v_j \rangle_{ L^2 }
& \le &
- \alpha \norm{v_j}^2_{L^2 }
  + C_1'
    \int_{-m}^m \abs{v_j(\tau)}^2 \,  d \tau 
\end{array}
\end{equation}
for some $C_1'> 0$.
Using the basic inequality
\begin{equation}
xy =(\sqrt{\alpha}x)(y/\sqrt{\alpha}) \le \frac{\alpha}{2} x^2 + \frac{1}{2\alpha} y^2,
\end{equation}
we arrive at
\begin{equation}
\label{eq:lem:bates:prf:mainProp:bndOnOuterV}
\begin{array}{lcl}
C_1'  \int_{-m}^m \abs{v_j(\tau)}^2 \, d \tau
& \ge &
\alpha \norm{v_j}^2_{ L^2  }
  - \norm{y_j}_{L^2  } \norm{v_j}_{L^2  }
\\[0.2cm]
& \ge & \frac{\alpha}{2} \norm{v_j}^2_{ L^2  }
  - \frac{1}{2\alpha} \norm{y_j}^2_{L^2  }.
\\[0.2cm]
\end{array}
\end{equation}
Multiplying
the first inequality in
\eqref{eq:bates:deriv:up:bnds}
by $\frac{\alpha}{2(1 + K_2)}$,
we find
\begin{equation}
\label{eq:lem:linop:prf:mainProp:restated:deriv:est}
0 \ge \frac{\alpha}{2(1 + K_2)}
\norm{v'_j}_{L^2}^2
- \frac{\alpha K_2}{2(1 + K_2)} \big( \norm{v_j}_{L^2}^2
+ \norm{y_j}_{L^2}^2 \big).
\end{equation}
Adding \eqref{eq:lem:bates:prf:mainProp:bndOnOuterV}
and \eqref{eq:lem:linop:prf:mainProp:restated:deriv:est},
we may use
the identity
\begin{equation}
\frac{\alpha}{2}
- \frac{\alpha K_2}{2(1 + K_2)}
= \frac{\alpha}{2(1 + K_2)},
\end{equation}
to obtain
\begin{equation}
\begin{array}{lcl}
C_1'
  \int_{-m}^m \abs{v_j(\tau)}^2 \, d \tau
& \ge &
\frac{\alpha}{2(1 + K_2)}
\big[ \norm{v_j}_{L^2}^2 + \norm{v'_j}_{L^2}^2 \big]
  - C_2' \norm{y_j}^2_{L^2  }
\\[0.2cm]
& = & \frac{\alpha}{2(1 + K_2)}
  - C_2' \norm{y_j}^2_{L^2  }
\\[0.2cm]
\end{array}
\end{equation}
for some $C_2'> 0$.
In view of the bound
\begin{equation}
  \limsup_{j \to \infty}
  \norm{y_j}^2_{ L^2 }  \le \mu(\delta)^2,
\end{equation}
the strong convergence
$v_j \to V_* \in L^2([-M, M])$
implies that
\begin{equation}
\label{eq:linop:strongConvergenceCompactSubsets:resultingBounds:ge}
\norm{V_*}^2_{H^1 } \ge
 [C_1']^{-1}
 \big[\frac{\alpha}{2(1 + K_2)}
  - C_2' \mu(\delta)^2 \big],
\end{equation}
as desired. The bound for $W_*$
follows in a very similar fashion.
\end{proof}

\begin{proof}[Proof of Proposition
\ref{prp:bates:lowBoundOnE}]
For any $0 < \delta < \delta_0$ and $k \ge 1$, Lemma's
\ref{lem:bates:approx:upBnd} and
\ref{lem:bates:approx:lowBnd} show that
the function $V_*$ defined in
Lemma \ref{lem:bates:approxSetting}
satisfies
\begin{equation}
\label{eq:linop:strongConvergenceCompactSubsets:resultingBounds:final}
K_1^2 \mu(\delta)^2 \ge
\norm{V_*}^2_{H^1 }
 \ge K_3 - K_4 \mu(\delta)^2,
\end{equation}
which gives
$\big(K_1^2 + K_4\big) \mu(\delta)^2 \ge K_3 > 0$,
as desired.
The same computation works for
$\mu^{\mathrm{adj}}$.
\end{proof}

\section{Nonlinear bounds}
\label{sec:nl}

In this section we establish Proposition 
\ref{prop:fxp:reformulation:exist} by
obtaining bounds on the nonlinearities $\mathcal{R}_A$
and $\mathcal{R}_B$. The computations are relatively
straightforward and included for completeness.

\begin{lemma}
\label{lem:nl:A}
Consider the setting of Theorem \ref{thm:mr:main}
and Proposition \ref{prop:fxp:reformulation:exist}.
There exists $K >0$ so that
for any $0 < \mu < 1$,
any $0 < h < 1$ and any $a \in \mathcal{A}_\diamond$,
the estimate
\begin{equation}
\begin{array}{lcl}
\norm{\mathcal{R}_A(\overline{c}, v; a) }_{L^2}
 &\le&
    K \mu^2 
\\[0.2cm]
\end{array}
\end{equation}
holds for each $(\overline{c},v) \in \mathcal{Z}_{\mu}$,
while the estimate
\begin{equation}
\begin{array}{lcl}
\norm{\mathcal{R}_A(\overline{c}^{(2)}, v^{(2)};a)
  - \mathcal{R}_A(\overline{c}^{(1)}, v^{(1)};a)
 }_{L^2}
 &\le&
   K \mu
    \norm{
      (\overline{c}^{(2)} - \overline{c}^{(1)}, v^{(2)} - v^{(1)} )
    }_{\mathbb{R} \times H^1 }
\\[0.2cm]
\end{array}
\end{equation}
holds for each set of pairs
$(\overline{c}^{(1)},v^{(1)}) \in \mathcal{Z}_{\mu}$
and $(\overline{c}^{(2)},v^{(2)}) \in \mathcal{Z}_{\mu}$.
\end{lemma}
\begin{proof}
The first term in $\mathcal{R}_A$ can be handled
by the elementary estimates
\begin{equation}
\begin{array}{lcl}
\norm{\overline{c} v'}_{L^2} & \le &
 \abs{\overline{c} } \norm{v}_{H^1},
\\[0.2cm]
\norm{\overline{c}^{(2)} [v^{(2)}]'
- \overline{c}^{(1)}[v^{(1)}]' }_{L^2}
& \le & \abs{\overline{c}^{(2)} - \overline{c}^{(1)} }
          \norm{v^{(2)}}_{H^1}
        + \abs{\overline{c}^{(1)} }
          \norm{v^{(1)} - v^{(2)} }_{H^1}.
\end{array}
\end{equation}
Writing    $ \mathcal{N}(v;a)  = g(\Phi_{*;a}+v;a) -  g(\Phi_{*;a};a) - g_u(\Phi_{*;a};a)v$, 
we obtain the pointwise bounds
\begin{equation}
\begin{array}{lcl}
    |\mathcal{N}( v;a)(\xi)|  &\le & C' \abs{v(\xi)}^2,
\\[0.2cm]
    |\mathcal{N}( v^{(1)};a)(\xi) 
    - \mathcal{N}(v^{(2)};a)(\xi)| 
    &\le & C' \big[\abs{v^{(1)}(\xi)} +\abs{v^{(2)}(\xi)} \big] \abs{v^{(1)}(\xi) - v^{(2)}(\xi) }
\end{array}
\end{equation}
for some $C'> 0$
as a consequence of the a-priori bounds on $\norm{v}_\infty$, 
$\norm{v}^{(1)}_\infty$ and $\norm{v}^{(2)}_\infty$. In particular,
we find
\begin{equation}
\begin{array}{lcl}
    \norm{ \mathcal{N}(v;a) }_{L^2}
    & \le & C' \norm{v}_{H^1} \norm{v}_{L^2},
    \\[0.2cm]
    \norm{ \mathcal{N}( v^{(1)};a)
    - \mathcal{N}( v^{(2)};a)}_{L^2} 
    &\le& C' \big[ \norm{v^{(1)}}_{H^1} + \norm{v^{(2)}}_{H^1} + \norm{v^{(1)} - v^{(2)}}_{L^2} \big].
\end{array}
\end{equation}
The desired bounds follow
readily from these estimates.
\end{proof}

\begin{lemma}
\label{lem:nl:B}
Consider the setting of Theorem \ref{thm:mr:main}
and Proposition \ref{prop:fxp:reformulation:exist}.
There exists $K > 0$
so that for each $0 < h < 1$,
every $k \ge 1$ and each $a \in \mathcal{A}_\diamond$
we have the bound
\begin{equation}
\begin{array}{lcl}
\norm{
  \mathcal{R}_B(h,k;a)
}_{L^2}
  & \le & K h .
\\[0.2cm]
\end{array}
\end{equation}
\end{lemma}
\begin{proof}
In view of \eqref{eq:m:h:k:ineq:v:vp}, this follows readily
from the exponential decay of the functions $\Phi_{*;a}'''$
and the compactness of $\mathcal{A}_\diamond$.
\end{proof}

\begin{proof}[Proof of Proposition \ref{prop:fxp:reformulation:exist}]
Upon writing
\begin{equation}
    \mathcal{T}_{h,k;a}(\overline{c}, v) = [ \beta_{h,k;a}, \mathcal{V}_{h,k;a}] \big( \mathcal{R}_A(\overline{c}, v;a) + \mathcal{R}_B(h,k;a)\big),
\end{equation}
Lemma's \ref{lem:nl:A}-\ref{lem:nl:B}
provide the bounds
\begin{equation}
\begin{array}{lcl}
    \norm{T_{h,k;a} (\overline{c}, v) }_{\mathbb{R} \times H^1} 
    & \leq  & C' [\mu^2 + h ] ,
     \\[0.2cm]
    \norm{\mathcal{T}_{h,k;a} (\overline{c}^{(1)}, v^{(1)}) - \mathcal{T}_{h,k;a} (\overline{c}^{(2)}, v^{(2)} ) }_{\mathbb{R} \times H^1} 
    & \leq &
      C' \mu 
    \norm{
      (\overline{c}^{(2)} - \overline{c}^{(1)}, v^{(2)} - v^{(1)} )
    }_{\mathbb{R} \times H^1}
\end{array}
\end{equation}
for some $C'> 0$.
In particular, upon taking $\mu = 2 C' h$ and $h$ sufficiently small,
we see that $\mathcal{T}_{h,k;a}$ maps $\mathcal{Z}_\mu$ to $\mathcal{Z}_{\mu}$ and is a contraction, which yields
the result.
\end{proof}

\bibliographystyle{klunumHJ}
\bibliography{ref}

\end{document}

%% file: packages.tex
\usepackage[T1]{fontenc}
\usepackage[utf8]{inputenc}

\usepackage{multirow}
\usepackage{pdfpages}

\usepackage[margin=15pt,font=small,labelfont={bf,sf}]{caption}
\usepackage[font=small,labelfont={bf}]{caption}
\usepackage{epigraph}
\usepackage[explicit]{titlesec}

\usepackage{afterpage}
\usepackage{mathtools}
\usepackage{amsfonts,amssymb,amsmath,amsthm}
\usepackage{titlesec}

\usepackage{enumerate}
\usepackage[shortlabels]{enumitem}
\usepackage{pinlabel}
\usepackage{caption}
\usepackage{subcaption}

\usepackage{verbatim}
\usepackage[normalem]{ulem}
\usepackage{bbm}

\usepackage{tikz}
\usepackage{enumerate}
\usepackage{cases}
\usepackage{multicol}
\usepackage{sidecap}
\usepackage{float}
\usepackage{xcolor}
\definecolor{blue2}{rgb}{0.67, 0.9, 0.93}
\usepackage[color=blue2]{todonotes}
\usepackage{cases}

\usepackage[normalem]{ulem}
\usepackage[hidelinks=true]{hyperref}

\numberwithin{equation}{section}

\newtheorem{theorem}{Theorem}[section]
\newtheorem{lemma}[theorem]{Lemma}
\newtheorem{proposition}[theorem]{Proposition}
\newtheorem{corollary}[theorem]{Corollary}
	
\theoremstyle{definition}

\newcommand{\R}{{\mathbb R}}

\newcommand{\Z}{\mathbb{Z}}

\usepackage{setspace}
\usepackage{array}

\newcommand\norm[1]{\left|\left|{#1}\right|\right|}

\usepackage[a4paper, total={6in,9in}]{geometry}
